\documentclass[11pt,reqno, a4paper]{amsart}
\usepackage{amsmath}
\usepackage{pifont}
\usepackage{}
\usepackage{amsfonts}
\usepackage{amssymb}
\usepackage{txfonts}
\usepackage{mathrsfs}
\usepackage{bbm}
\usepackage{enumerate}
\usepackage{color}

\theoremstyle{plain}
\newtheorem{theorem}{Theorem}[section]
\newtheorem{remark}{Remark}[section]
\newtheorem{lemma}{Lemma}[section]
\newtheorem{definition}{Definition}[section]
\newtheorem{proposition}{Proposition}[section]

\newtheorem{Example}{Example}[section]
\numberwithin{equation}{section}
\begin{document}

\title{Sharp regularity and Cauchy problem of the spatially homogeneous Boltzmann equation with Debye-Yukawa potential}

\author{L\'eo Glangetas,\,\,Hao-Guang Li}
\address{L\'eo Glangetas,
\newline\indent
Universit\'e de Rouen, CNRS UMR 6085, Math\'ematiques
\newline\indent
76801 Saint-Etienne du Rouvray, France}
\email{leo.glangetas@univ-rouen.fr}
\address{Hao-Guang Li,
\newline\indent
School of mathematics and statistics, South-central university for nationalities 430074,
\newline\indent
Wuhan, P. R. China}
\email{lihaoguang@mail.scuec.edu.cn}



\begin{abstract}
In this paper, we study the Cauchy problem for the linear spatially homogeneous Boltzamnn equation with Debye-Yukawa potential.
Using the spectral decomposition of the linear operator,
we prove that, for an initial datum in the sense of distribution which contains the dual of the Sobolev spaces, there exists a unique solution which belongs to a more regular Sobolev space for any positive time. We also study the sharp regularity of the solution.
\end{abstract}

\maketitle
\section{Introduction and main results}
In this work, we consider the spatially homogeneous Boltzmann equation
\begin{equation}\label{eq1.10}
\frac{\partial f}{\partial t} = Q(f,f)
\end{equation}
where $f=f(t,v)$ is the density distribution function depending only on two varia\-bles $t\geq0$~and $v\in\mathbb{R}^{3}$. The Boltzmann bilinear collision operator is given by
\begin{equation*}
Q(g,f)(v)=\int_{\mathbb{R}^{3}}\int_{S^{2}}B(v-v_{\ast},\sigma)\left(g(v_{\ast}^{\prime})f(v^{\prime})-g(v_{\ast})f(v)\right)dv_{\ast}d\sigma,
\end{equation*}
where for $\sigma\in \mathbb{S}^{2}$,~the symbols~$v_{\ast}^{\prime}$~and~$v^{\prime}$~are abbreviations for the expressions,
$$
v^{\prime}=\frac{v+v_{\ast}}{2}+\frac{|v-v_{\ast}|}{2}\sigma,\,\,\,\,\, v_{\ast}^{\prime}
=\frac{v+v_{\ast}}{2}-\frac{|v-v_{\ast}|}{2}\sigma,
$$
which are obtained in such a way that collision preserves momentum and kinetic energy,~namely
$$
v_{\ast}^{\prime}+v^{\prime}=v+v_{\ast},\,\,\,\,\, |v_{\ast}^{\prime}|^{2}+|v^{\prime}|^{2}=|v|^{2}+|v_{\ast}|^{2}.
$$
For monatomic gas, the collision cross section $B(v-v_{\ast},\sigma)$ is a non-negative function which~depends only on $|v-v_{\ast}|$ and $\cos\theta$
which is defined through the scalar product in $\mathbb{R}^{3}$ by
$$\cos\theta=\frac{v-v_{\ast}}{|v-v_{\ast}|}\cdot\sigma.$$
Without loss of generality, we may assume that $B(v-v_{\ast},\sigma)$ is supported on the set $\cos\theta\geq0,$ i.e. where $0\leq \theta\leq\frac{\pi}{2}$.
See for example \cite{NYKC1}, \cite{Villani} for more explanations about the support of $\theta$.
For physical models, the collision cross section usually takes the form
\begin{equation*}
B(v-v_{\ast},\sigma)=\Phi(|v-v_{\ast}|)b(\cos\theta),
\end{equation*}
with a kinetic factor
$$\Phi(|v-v_{\ast}|)=|v-v_{\ast}|^{\gamma},\,\,\gamma\in]-3,+\infty[.$$
The molecules are said to be Maxwellian when the parameter $\gamma=0$.

Except for the hard sphere model, the function $b(\cos\theta)$ has a singularity at $\theta=0.$
For instance, in the important model case of the inverse-power potentials,
$$\phi(\rho)=\frac{1}{\rho^r}, \,\,\text{with}\,\, r>1,$$
with $\rho$ being the distance between two interacting particles
in the physical 3-dimensional space $\mathbb{R}^{3}$,
\begin{equation*}
b(\cos\theta)\sin\theta\approx K\theta^{-1-\frac{2}{r}}, \,\,\text{as}\,\,\theta\rightarrow0^+,
\end{equation*}
The notation $a\approx b$ means that there exist positive constants $C_2>C_1>0$, such that
$$C_1 \, a\leq b\leq C_2 \, a.$$
Notice that the Boltzmann collision operator is not well defined for the case when $r=1$ corresponding to the Coulomb potential.

If the inter-molecule potential satisfies the Debye-Yukawa type potentials where the potential function is given by
$$\phi(\rho)=\frac{1}{\rho\, e^{\rho^s}},\,\,\text{with}\,\,s>0,$$
then the collision cross section has a singularity in the following form
\begin{equation}\label{b}
b(\cos\theta)\approx \theta^{-2}(\log\theta^{-1})^{\frac{2}{s}-1},\,\,\text{when}\,\,\theta\rightarrow 0^+,  \,\,\text{with}\,\,s>0.
\end{equation}
This explicit formula was first appeared in the Appendix in \cite{YSCT}.
In some sense, the Debye-Yukawa type potentials is a model between the Coulomb potential corresponding to $s=0$ and the inverse-power potential. For further details on the physics background and the derivation of the Boltzmann equation, we refer the reader to the extensive expositions \cite{Cerci}, \cite{Villani}.

We linearize the Boltzmann equation near the absolute Maxwellian distribution
$$
\mu(v)=(2\pi)^{-\frac 32}e^{-\frac{|v|^{2}}{2}}.
$$
Let $f(t,v)=\mu(v)+\sqrt{\mu}(v)g(t,v)$. Plugging this expression into $\eqref{eq1.10}$, we have
$$
\frac{\partial g}{\partial t}+\mathcal{L}[g]={\bf \Gamma}(g, g)
$$
with
$$
{\bf \Gamma}(g, h)=\frac{1}{\sqrt{\mu}}Q(\sqrt{\mu}g,\sqrt{\mu}h),\,\,
\mathcal{L}(g)=-\frac{1}{\sqrt{\mu}}[Q(\sqrt{\mu}g,\mu)+Q(\mu,\sqrt{\mu}g)].
$$
Then the Cauchy problem \eqref{eq1.10} can be re-writed in the form
\begin{equation*}
\left\{ \begin{aligned}
         &\partial_t g+\mathcal{L}(g)={\bf \Gamma}(g, g),\,\\
         &g|_{t=0}=g_{0}.
\end{aligned} \right.
\end{equation*}
In the present work, we consider the linearized Cauchy problem
\begin{equation} \label{eq-1}
\left\{ \begin{aligned}
         &\partial_t g+\mathcal{L}(g)=0,\,\\
         &g|_{t=0}=g_{0}.
\end{aligned} \right.
\end{equation}

In the case of the inverse-power potential with $r>1$, 
the regularity of the Boltzmann equation has been studied by numerous papers.  
It is well known that the non-cutoff spatially homogeneous Boltzmann equation
enjoys an $\mathscr{S}(\mathbb{R}^3)$-regularizing effect 
for the weak solutions to the Cauchy problem (see \cite{DW,YSCT}).
Regarding the Gevrey regularity, Ukai showed in \cite{Ukai} 
that the Cauchy problem for the Boltzmann equation 
has a unique local solution in Gevrey classes. 
Then, Desvillettes, Furioli and Terraneo proved in \cite{DFT}
the propagation of Gevrey regularity for solutions 
of the Boltzmann equation with Maxwellian mole\-cules. 
For mild singularities, Morimoto and Ukai proved in \cite{MU} 
the Gevrey regularity of smooth Maxwellian decay solutions 
to the Cauchy problem of the spatially homogeneous Boltzmann equation 
with a modified kinetic factor. 
See also \cite{TZ} for the non-modified case.
On the other hand, Lekrine and Xu proved in \cite{L-X} 
the property of Gevrey smoothing effect for the weak solutions 
to the Cauchy problem associated 
to the radially symmetric spatially homogeneous Boltzmann equation 
with Maxwellian molecules for $r>2$. 
This result was then completed by Glangetas and Najeme 
who established in \cite{G-N} the analytic smoothing effect 
in the case when $1<r<2$.
In \cite{ NYKC1, HAOLI}, the solutions of the Cauchy problem \eqref{eq-1} 
for linearized non-cutoff Boltzmann equation belong 
to the symmetric Gelfand-Shilov spaces $S^{r/2}_{r/2}(\mathbb{R}^3)$
for any positive time and
$$
\|e^{ct \mathcal{H}^{\frac{1}{r}}}g(t)\|_{L^2}\leq\,C\|g_0\|_{L^2},
$$
where 
\begin{equation*}
  \mathcal{H}=-\triangle +\frac{|v|^2}{4}.
\end{equation*}
The Gelfand-Shilov space $S^{\nu}_{\nu}(\mathbb{R}^3)$
with $\nu\geq\frac{1}{2}$ can be identify with
$$
S^{\nu}_{\nu}(\mathbb{R}^{3})=\left\{f\in C^\infty (\mathbb{R}^3);  \exists \tau>0,
\|e^{\tau \, \mathcal{H}^{\frac{1}{2\nu}}}f\|_{L^2}<+\infty\right\}.
$$
This space can also be characterized as
is the space of smooth functions $f\in\,C^{+\infty}(\mathbb{R}^3)$ satisfying
(see Appendix \ref{Appendix}):
$$
\exists\, A>0,\, C>0,\,
\sup_{v\in\mathbb{R}^3}|v^{\beta}\partial^{\alpha}_vf(v)|
\leq\,CA^{|\alpha|+|\beta|}
(\alpha!)^{\nu}(\beta!)^{\nu},\,\,\forall\,\alpha,\,\beta\in\mathbb{N}^3.
$$

The linear Boltzmann operator $\mathcal{L}$ is shown to be diagonal in the basis of $L^2(\mathbb{R}^3)$ and this property has been used in  \cite{NYKC3} and \cite{GLX_2015} to prove that the Cauchy problem to the non-cutoff spatially homogeneous Boltzmann equation with the small initial datum $g_0\in L^2(\mathbb{R}^3)$ has a global solution, which belongs to the Gelfand-Shilov class $S^{r/2}_{r/2}(\mathbb{R}^3)$.

The initial value problem in a space of probability
measures defined via the Fourier transform
has been studied in \cite{CK} and \cite{Morimoto2}.
Recently, the case of initial datum in the sense of distribution,
which contains the dual space of a Gelfand-Shilov class
for the linear case has been studied in \cite{HAOLI}.

In this paper, we consider the collision kernel in the Maxwellian molecules case and the angular function $b$ satisfying the Debye-Yukawa potential \eqref{b} for some $s>0$.
For convenience, we denote
\begin{equation}\label{beta}
\beta(\theta)=2\pi b(\cos\theta)\sin\theta.
\end{equation}
We study the smoothing effect for the Cauchy problem \eqref{eq-1} associated to the non-cutoff
spatially homogeneous Boltzmann equation with Debye-Yukawa potential \eqref{b}.
The singularity of the collision kernel $b$
endows the linearized Boltzmann operator $\mathcal{L}$ with
the logarithmic Gelfand-Shilov regularity property (see Proposition \ref{regular}).
The logarithmic regularity theory was first introduced in \cite{Morimoto} on the hypoellipticity of the infinitely degenerate elliptic operator and was developed in \cite{Morimoto-Xu1},\cite{Morimoto-Xu2} on the logarithmic Sobolev estimates.
Recently, for $0<s<2$, the initial datum $f_0\geq0$ and
$$\int_{\mathbb{R}^3}f_0(v)(1+|v|^2+\log(1+f_0(v)))dv<+\infty,$$
Morimoto, Ukai, Xu and Yang in \cite{YSCT} show that 
the weak solution of the Cauchy problem \eqref{eq1.10} satisfying
$$\sup_{t\in[0,T]}\int_{\mathbb{R}^3}f(t,v)(1+|v|^2+\log(1+f(t,v)))dv<+\infty$$
enjoys a $H^{+\infty}(\mathbb{R}^3)$ smoothing effect property.

In order to precise the regularity of the solution of the Cauchy problem,
we introduce some functional spaces.
The linear operator $\mathcal{L}$ is nonnegative (\cite{NYKC1,NYKC2,NYKC3}),\,
with the null space
$$
\mathcal{N}=\text{span}\left\{\sqrt{\mu},\,\sqrt{\mu}v_1,\,\sqrt{\mu}v_2,\,
\sqrt{\mu}v_3,\,\sqrt{\mu}|v|^2\right\}.
$$
Denote by $\mathbf{P}$ the orthoprojection from $L^2(\mathbb{R}^{3})$
into $\mathcal{N}$.\,\,Then
$$
(\mathcal{L}g,\,g)=0\Leftrightarrow g=\mathbf{P}g
$$
and the operator $\mathcal{L}+\mathbf{P}$ is a positive self-adjoint operator.
We define $$D(\mathcal{L})=\bigcup_{\tau>0}D_{\tau}(\mathcal{L})$$
with
\begin{equation*}
D_{\tau}(\mathcal{L})=\left\{u\in C^\infty(\mathbb{R}^3),\,\,
\sum^{+\infty}_{k=0}\tau^k(k!)^{-1}\|(\mathcal{L}+\mathbf{P})^\frac{k}{2}u\|^2_{L^2}<+\infty\right\},
\end{equation*}
which is a Banach space with the norm
$$
\|u\|^2_{D_{\tau}(\mathcal{L})}=\sum^{+\infty}_{k=0}\tau^k(k!)^{-1}\|(\mathcal{L}+\mathbf{P})^\frac{k}{2}u\|^2_{L^2}.
$$
Analogously, we define
$$D^+(\mathcal{L})=\bigcup_{\tau>0}D^+_{\tau}(\mathcal{L})$$
where
$$
D^+_{\tau}(\mathcal{L})=\left\{u\in C^\infty(\mathbb{R}^3),\,\,\sum^{+\infty}_{k=0}\tau^k(k!)^{-1}\|(\mathcal{L}+\mathbf{P})^\frac{k+1}{2}u\|^2_{L^2}<+\infty\right\}
$$
with the norm
$$\|u\|^2_{D^+_{\tau}(\mathcal{L})}=\sum^{+\infty}_{k=0}\tau^k(k!)^{-1}\|(\mathcal{L}+\mathbf{P})^\frac{k+1}{2}u\|^2_{L^2}.$$
The distribution spaces
$(D(\mathcal{L}))',\,
(D^+(\mathcal{L}))'
$
are defined as follows
 $$(D(\mathcal{L}))'=\bigcup_{\tau>0}(D_{\tau}(\mathcal{L}))'
 ,\quad (D^+(\mathcal{L}))'=\bigcup_{\tau>0}(D^+_{\tau}(\mathcal{L}))'
 $$
where $(D_{\tau}(\mathcal{L}))'$, $(D^+_{\tau}(\mathcal{L}))'$ are the dual spaces of $D_{\tau}(\mathcal{L})$ and $D^+_{\tau}(\mathcal{L})$.
Now we begin to present our result.
Firstly, we give the definition of the weak solution 
of the Cauchy problem \eqref{eq-1}.
\begin{definition} For any $g_0\in \left(D(\mathcal{L})\right)'$, $T>0$,
g(t,v) is a weak solution of the Cauchy problem \eqref{eq-1} if\,\,$\mathbf{P}g\equiv\mathbf{P}g_0,$
\begin{align}\label{solution}
&g\in L^{\infty}\bigl( [0,T], \left(D(\mathcal{L})\right)' \bigr)
\cap 
H^1\bigl([0,T],\left(D^+(\mathcal{L})\right)'\bigr),\nonumber
\\
&\mathcal{L}^{\frac{1}{2}} g\in L^{2}\left([0,T],\left(D(\mathcal{L})\right)'\right),
\end{align}
and for any $\phi\in C^1\bigl([0,T],\,\, C_{0}^\infty(\mathbb{R}^{3})\bigr)$
we have
\begin{align}\label{def}
\forall t\in[0,T],\quad
&\langle g(t),\phi(t)\rangle-\langle g_0,\phi(0)\rangle\nonumber\\
&=\int^{t}_{0}\langle g,\partial_{\tau}\phi\rangle d\tau
- \int^{t}_{0}\langle g,\mathcal{L}\phi\rangle d\tau.
\end{align}
\end{definition}
In the main theorem, we consider the initial distribution data case, 
which is given in the following.
\begin{theorem}\label{trick}
Assume that the Debye-Yukawa potential $b(\,\cdot\,)$ 
is given in~$\eqref{b}$ with $s>0$
and $g_0\in\, \left(D(\mathcal{L})\right)'$.
Therefore the Cauchy problem \eqref{eq-1}~admits a unique weak solution.
Moreover there exist $t_0>0$ and $c_0>0$ such that for all $ t\ge t_0/c_0$
\begin{align*}
&\|e^{- t_0 \, \left(\log(e+\mathcal{H})\right)^{\frac{2}{s}}}g_0 \|_{L^2} < +\infty,\nonumber
\\
&\|e^{tc_0\left(\log(\mathcal{H}+e)\right)^{\frac{2}{s}}}
(\mathbf{I}-\mathbf{P})g(t) \|_{L^2}
\leq
e^{-\frac{1}{4}\lambda_{2,0}t}
\|e^{- t_0 \, \left(\log(e+\mathcal{H})\right)^{\frac{2}{s}}}
  (\mathbf{I}-\mathbf{P}) g_0 \|_{L^2}
\end{align*}
where $\mathcal{H}=-\Delta+{\textstyle \frac{|v|^2}{4}}$,
$$\lambda_{2,0}=\int^{\pi/4}_0\beta( \theta)(1-\sin^4\theta-\cos^4\theta)d\theta>0.$$
In particular, for all $t>t_0/c_0$,
the solution $g(t)$ belongs to $L^2(\mathbb{R}^3)$.
\end{theorem}

\begin{remark}
The regularizing effect $g(t)\in L^2(\mathbb{R}^3)$ has usually a positive delay time.
For example, consider some $\tau_0>0$ and the initial data
$g_0=\sum_{n\geq 1} \frac{1}{n}  e^{\tau_0 \lambda_{n,0}} \phi_{n,0,0}$.
It is easy to check that $g(t)\in L^2(\mathbb{R}^3)$ for $t\geq\tau_0$
but $g(t)\not\in L^2(\mathbb{R}^3)$ for $t<\tau_0$.
\end{remark}

In order to precise the regularizing effect in the Sobolev spaces,
it is convenient to consider the symmetric weighted Sobolev space 
$Q^{2\tau}(\mathbb{R}^3)$ introduced by Shubin \cite{Shubin} with norm
\[\|u\|_{Q^{2\tau}(\mathbb{R}^3)} =
\Bigl\|\Bigl( -\Delta+{\textstyle \frac{|v|^2}{4} + e} \Bigr)^{\tau} \, u\Bigr\|_{L^2}.\]
\begin{theorem}\label{trick2}
Regularizing effect for an initial data in $L^2(\mathbb{R}^3)$.
\\
Assume that the Debye-Yukawa potential $b(\,\cdot\,)$ is given in~$\eqref{b}$ with $s>0$
and $g_0\in L^2(\mathbb{R}^3)$.
Therefore the Cauchy problem \eqref{eq-1}~admits  a unique weak solution.
Moreover there exists  $c_0>0$ such that :
\\
1) Case $0<s\leq2$.
\begin{equation}\label{rate1}
\forall t>0,\quad
\Bigl\|\Bigl(e+\mathcal{H}\Bigr)^{c_0 t} \,
(\mathbf{I}-\mathbf{P})g(t)\Bigr\|_{L^2}\
\leq
 e^{-\lambda_{2,0}t}\|(\mathbf{I}-\mathbf{P})g_{0}\|_{L^2(\mathbb{R}^3)}.
\end{equation}
This shows that $g(t)$
belongs to the Sobolev space $Q^{ct}(\mathbb{R}^3)$ for any time $t>0$.
\\
2) Case $0<s<2$.
there exists a constant $c_s>0$ such that for any $t>0$, one has
\begin{equation}\label{rate2}
\forall k\geq 0,\quad
\|(\mathbf{I}-\mathbf{P})g(t)\|_{Q^{k}}
\leq
 e^{-\lambda_{2,0}t}e^{c_s \, (1/t)^{\frac{s}{2-s}}  \, k^{\frac{2}{2-s}}}\|(\mathbf{I}-\mathbf{P})g_{0}\|_{L^2(\mathbb{R}^3)}.
\end{equation}
\end{theorem}

\begin{remark}\label{remark} Comments on the regularizing effect.
When the singularity of the collision cross section \eqref{b} 
for $\theta$ near 0 become smoother (that is when the real $s$ increases), 
the regularizing effect become weaker, and disappears in the context of the Sobolev spaces when $s>2$ :
\\
- Case $0<s<2$.
The solution $g(t) \in \cap_{k\geq 0} Q^{k}(\mathbb{R}^3)$ for each positive time.
\\
- Case $s=2$. The regularizing effect in $Q^{k}(\mathbb{R}^3)$ has usually a positive time delay.
For example, consider the initial data
$g_0=\sum_{n\geq 2} \frac{1}{n^{\frac12} \log n}   \varphi_{n,0,0}$
where $\left\{\varphi_{n,l,m}(v)\right\}$ constitutes an orthonormal basis of $L^2(\mathbb{R}^3)$, which is given in Section \ref{S2}.
We can check that there exists $t_k>0$ such that $g(t)\not\in Q^{k}(\mathbb{R}^3)$ for $0\leq t<t_k$.
\\
- Case $s>2$. There is no regularizing effect in the Sobolev space.
Consider any real numbers $0<\tau<\tau'$ and
$g_0=\sum_{n\geq 2} \frac{1}{n^{\frac{\tau+1}2} \log n} \varphi_{n,0,0}$
where $\varphi_{n,0,0}$ is given in Section \ref{S2}.
We can check that for $t\geq0$ the solution $g(t)$ stays in the space $Q^{\tau}(\mathbb{R}^3)$,
but never belongs to $Q^{\tau'}(\mathbb{R}^3)$.
However, there is a very slight regularizing effect and the Boltzmann equation remains irreversible.
\end{remark}
{\color{black} 
\begin{remark}\label{remark non-linear} 
We think that the non-linear case is similar to the linear case,
but the proofs are more technical. This work is a first step to study the non-linear case.
\end{remark}
}
The rest of the paper is arranged as follows. In Section \ref{S2}, we introduce the spectral analysis of the linear Boltzmann operator and in Section \ref{S3} we precise some properties of the distribution spaces.
The proof of the main Theo\-rems~\ref{trick}-\ref{trick2} will be presented in Section \ref{S4}, where we construct
a sequence of solutions of the Cauchy problem \eqref{eq-1} with initial datum equal to the projection of $g_0$ on an increasing sequence of finite dimensional subspaces,
which converges to the solution of the Cauchy problem.
In the appendix, we present some spectral properties
of the functional spaces used in this paper and the proof of some technical lemmas.

\section{The preliminary results }\label{S2}
We first recall the spectral decomposition of linear Boltzmann operator.
In the cutoff case, that is, when $b(\cos\theta)\sin\theta\in\,L^1([0,\frac{\pi}{2}])$, it was shown in \cite{WU} that
\begin{equation*}
\mathcal{L}(\varphi_{n, l, m})=\lambda_{n,l}\, \varphi_{n, l, m}, 
\quad n,l\in\mathbb{N},\,\,m\in\mathbb{Z},\,\, |m|\leq l
\end{equation*}
where
\begin{equation}\label{l}
\lambda_{n,l}=\int^{\frac{\pi}{4}}_{0}\beta(\theta)\Big(1+\delta_{n, 0}\delta_{l, 0}
-(\sin\theta)^{2n+l}P_{l}(\sin\theta)-(\cos\theta)^{2n+l}P_{l}(\cos\theta)\Big)d\theta.
\end{equation}
This diagonalization of the linearized Boltzmann operator with Maxwellian mole\-cu\-les holds
as well in the non-cutoff case, (see \cite{Boby,Cerci,Dole,NYKC1,NYKC2}).

The eigenfunctions are
\begin{equation*}
\varphi_{n,l,m}(v)=\left(\frac{n!}{\sqrt{2}\Gamma(n+l+3/2)}\right)^{1/2}
\left(\frac{|v|}{\sqrt{2}}\right)^{l}e^{-\frac{|v|^{2}}{4}}
L^{(l+1/2)}_{n}\left(\frac{|v|^{2}}{2}\right)Y^{m}_{l}\left(\frac{v}{|v|}\right)
\end{equation*}
where $\Gamma(\,\cdot\,)$ is the standard Gamma function: for any $x>0$,
$$\Gamma(x)=\int^{+\infty}_0t^{x-1}e^{-x}dx.$$
The $l^{th}$-Legendre polynomial~$P_{l}$ and the Laguerre polynomial $L^{(\alpha)}_{n}$~of order $\alpha$,~degree $n$ (see\,\cite{San})\,read,
\begin{align*}
&P_{l}(x)=\frac{1}{2^ll!}\frac{d^l}{dx^l}(x^2-1)^l,\,\,\text{where}\,|x|\leq1;\\
&L^{(\alpha)}_{n}(x)=\sum^{n}_{r=0}(-1)^{n-r}\frac{\Gamma(\alpha+n+1)}{r!(n-r)!
\Gamma(\alpha+n-r+1)}x^{n-r}.
\end{align*}
For any unit vector $\sigma=(\cos\theta,\sin\theta\cos\phi,\sin\theta\sin\phi)$~with $\theta\in[0,\pi]$~and~$\phi\in[0,2\pi]$,~the orthonormal basis of spherical harmonics~$Y^{m}_{l}(\sigma)$ is
\begin{equation*}
Y^{m}_{l}(\sigma)=N_{l,m}P^{|m|}_{l}(\cos\theta)e^{im\phi},\,\,|m|\leq l,
\end{equation*}
where the normalisation factor is given by
$$
N_{l,m}=\sqrt{\frac{2l+1}{4\pi}\cdot\frac{(l-|m|)!}{(l+|m|)!}}
$$
and $P^{|m|}_{l}$~is the associated Legendre functions of the first kind of order $l$ and degree $|m|$ with
\begin{equation*}
P^{|m|}_{l}(x)= (1-x^2)^\frac{|m|}{2}
\left(\frac{\mathrm{d}}{\mathrm{d}x}\right)^{|m|} P_{l}(x).
\end{equation*}
We recall from Lemma 7.2 in \cite{GLX_2015} that
$$\widehat{\sqrt{\mu}\varphi_{n,l,{m}}}(\xi)
=(-i)^l \, (2\pi)^\frac{3}{4} \,
\Biggl(\frac{1}{\sqrt{2}n!\Gamma(n+l+\frac{3}{2})}\Biggr)^\frac{1}{2} \,
\Biggl(\frac{|\xi|}{\sqrt{2}}\Biggr)^{2n+l} \,
e^{-\frac{|\xi|^2}{2}} \,
Y_l^{m}\Biggl(\frac{\xi}{|\xi|}\Biggr).$$
We can extend the spectral decomposition to the Debye-Yukawa potential case.
The family $\Big(Y^m_l(\sigma)\Big)_{l\geq0,|m|\leq\,l}$ constitutes an orthonormal basis of the space $L^2(\mathbb{S}^2,\,d\sigma)$ with $d\sigma$ being the surface measure on $\mathbb{S}^2$ (see \cite{Jones}, \cite{JCSlater}).  
Noting that $\left\{\varphi_{n,l,m}(v)\right\}$ constitutes an orthonormal basis of $L^2(\mathbb{R}^3)$ composed of eigenvectors of the harmonic oscillator
(see\cite{Boby}, \cite{NYKC2})
\begin{equation*}
\mathcal{H}(\varphi_{n, l, m})=(2n+l+\frac 32)\, \varphi_{n, l, m}.
\end{equation*}
As a special case, $\left\{\varphi_{n, 0, 0}(v)\right\}$ 
constitutes an orthonormal basis of $L^2_{rad}(\mathbb{R}^3)$, 
the rad\-ial\-ly symmetric function space (see \cite{NYKC3}).
We have for suitable functions
$g$
\begin{equation}\label{Lg=Sum}
\mathcal{L}(g)=\sum^{\infty}_{n=0}\sum^{\infty}_{l=0}\sum^{l}_{m=-l}
\lambda_{n,l}\, g_{n,l,m}\, \varphi_{n, l, m}
\end{equation}
where $g_{n,l,m}=(g, \varphi_{n,l,m})_{L^2(\mathbb{R}^3)}$ and
\begin{equation*}
\mathcal{H}(g)=\sum^{\infty}_{n=0}\sum^{\infty}_{l=0}\sum^{l}_{m=-l}(2n+l+\frac 32)\, g_{n,l,m}\, \varphi_{n, l, m}\, .
\end{equation*}
Using this spectral decomposition, for any $s>0$, the definitions of $\left(\log(\mathcal{H}+e)\right)^s$,\, $e^{c\left(\log(\mathcal{H}+e)\right)^s}$ and $e^{c\mathcal{L}}$ are then classical.
\begin{remark}\label{2.2}
It is trivial to obtain from \eqref{l} that $\lambda_{0,0}=\lambda_{1,0}=\lambda_{0,1}=0$~and the others are strictly positive.\,\,
Thus the null space of the linear Boltzmann operator $\mathcal{L}$ is
$$
\mathcal{N}=\text{span}\left\{\sqrt{\mu},\,\sqrt{\mu}v_1,\,\sqrt{\mu}v_2,\,\sqrt{\mu}v_3,\,\sqrt{\mu}|v|^2\right\}.
$$
We have for $(n,l) \in\mathbb{N}^2$
\begin{align*}
\lambda_{n,l}
&=0 \ \ \ \ \ \text{if} \,\,  n+l\leq 1, \\
\lambda_{n,l} &>0 \ \ \ \ \ \text{otherwise}.
\end{align*}
\end{remark}
{\color{black}
We derive the following estimate of $\lambda_{n,l}$ defined in \eqref{l}.
\begin{proposition}
[Spectral estimates of the linearized Boltzmann operator $\mathcal{L}$]
\label{regular}
Let collision the kernel $b$ satisfies the Debye-Yukawa potential condition \eqref{b}.
Then, there exists a positive constant $c_{0}$ such that 
for any $n,l\in\mathbb{N}$, $n+l\geq 2$, we have 
\begin{equation}\label{c-0}
c_{0}\left(\log(2n+l+e)\right)^{\frac{2}{s}}\leq \lambda_{n,l}\leq \frac{1}{c_{0}}\left(\log(2n+l+e)\right)^{\frac{2}{s}},
\end{equation}
where e is the Natural constant.
\end{proposition}
}
\begin{proof}
From Remark \ref{2.2},
for any $n,l\in\mathbb{N}$ and  $n+l\geq 2$
we have $\lambda_{n,l}>0$.
Therefore we need only to consider the case $2n+l\rightarrow+\infty$.
We have from \eqref{beta} for $\theta\in]0,\frac{\pi}{4}]$
$$\beta(\theta)\approx
  (\sin\theta)^{-1}(\log(\sin\theta)^{-1})^{\frac{2}{s}-1}.$$
From \eqref{l} and putting
$x=\sin\theta$, $\lambda_{n,l}$ can be decomposed as follows
\begin{align}\label{dec}
\lambda_{n,l}
=&\int^{\frac{\pi}{4}}_{0}\beta(\theta)(1-P_{l}(\sin\theta)(\sin\theta)^{2n+l}-P_{l}(\cos\theta)(\cos\theta)^{2n+l})d\theta\nonumber\\
\approx&
\int^{\frac{\sqrt{2}}{2}}_{0}
(\log x^{-1})^{\frac{2}{s}-1} \,
\Bigl(1-P_{l}(x)\,x^{2n+l}
-P_{l}\Bigl(\sqrt{1-x^2}\Bigr) \,(1-x^2)^{\frac{2n+l}{2}}
\Bigr)\frac{dx}{x}
\nonumber\\
\approx&
\phantom{+}\int^{\frac{\sqrt{2}}{2}}_{0}
(\log x^{-1})^{\frac{2}{s}-1} \,
\Bigl(1-(1-x^2)^{\frac{2n+l}{2}}
\Bigr)\frac{dx}{x}
\nonumber\\
&-\int^{\frac{\sqrt{2}}{2}}_{0}
(\log x^{-1})^{\frac{2}{s}-1} \,
P_{l}(x)\,x^{2n+l} \frac{dx}{x}
\nonumber\\
&+ \int^{\frac{\sqrt{2}}{2}}_{0}
(\log x^{-1})^{\frac{2}{s}-1} \,
\Bigl(1-P_{l}\Bigl(\sqrt{1-x^2}\Bigr)\Bigr) \, (1-x^2)^{\frac{2n+l}{2}}
\frac{dx}{x}
\nonumber\\
=&A_{1}+A_{2}+A_{3}.
\end{align}
We first estimate $A_1$. Setting $y= x \, \sqrt{2n+l} $,
we decompose it in two parts
\begin{align*}
A_1 &=\int^{\frac{\sqrt{2}}{2}\sqrt{2n+l}}_{0}
\Bigl(\log \frac{\sqrt{2n+l}}{y} \Bigr)^{\frac{2}{s}-1} \,
\Bigl(1-\Bigl(1-\frac{y^2}{2n+l}\Bigr)^{\frac{2n+l}{2}}
\Bigr)\frac{dy}{y}
\\
&= \int_1^{\frac{\sqrt{2}}{2}\sqrt{2n+l}} +  \int_{0}^1
= A_{11} + A_{12}.
\end{align*}
The main term is $A_{11}$. Putting $z=\log \frac{\sqrt{2n+l}}{y}$, 
we get when $2n+l\to\infty$
\begin{align*}
A_{11} &=\int^{\log\sqrt{2n+l}}_{\log\sqrt{2}}
z^{\frac{2}{s}-1} \,
\Bigl(1-\Bigl(1-e^{-2z}\Bigr)^{\frac{2n+l}{2}}
\Bigr) dz
\\
&= \Bigl(\int^{\log\sqrt{2n+l}}_{\log\sqrt{2}}
z^{\frac{2}{s}-1} \, dz\Bigr)  \,
\Bigl(1 + O\Bigl(\frac{1}{2} \Bigr)^{\frac{2n+l}{2}}
\Bigr)
\\
&\approx \frac{s}{2} \, \bigl(\log\sqrt{2n+l}\bigr)^{\frac{2}{s}}.
\end{align*}
We now check that the other term $A_{12}$ has a lower order
($A_{2}$ and $A_{3}$ will have also a lower order).
We decompose the term $A_{12}$ as follows:
\begin{align*}
A_{12}
&=\int_{0}^1
\Bigl(\log {\sqrt{2n+l}} + \log \frac{1}{y}
\Bigr)^{\frac{2}{s}-1} \,
\Bigl(1-\Bigl(1-\frac{y^2}{2n+l}\Bigr)^{\frac{2n+l}{2}}
\Bigr)\frac{dy}{y}
\\
&=
\bigl(\log\sqrt{2n+l}\bigr)^{\frac{2}{s} - 1} \,
\int^{1}_{0} g_{2n+l}(y) \, dy
\end{align*}
where
$$g_k(y) =
\Bigl(1+\frac{1}{\log\sqrt{k}} \, \log \frac{1}{y}
\Bigr)^{\frac{2}{s}-1} \,
\Bigl(1-\Bigl(1-\frac{y^2}{k}\Bigr)^{\frac{k}{2}}\Bigr) \frac{1}{y}.$$
It is easy to check that, uniformly for $y\in]0,1]$
and $k\geq 2$, we have
$$g_k(y) = \max\Bigl(1,\Bigl(\log\frac{1}{y}\Bigr)^{\frac{2}{s}-1}\Bigr) \,O(y)
\quad\mathrm{and}\quad
g_k(y)\xrightarrow[k\to\infty]{}
\Bigl(1-e^{-\frac{1}{2}y^2}\Bigr) \frac{1}{y}.
$$
From the dominated convergence theorem we get
\begin{align*}
A_{12} &\approx \Bigl(\log\sqrt{2n+l}\Bigr)^{\frac{2}{s} - 1} \,
\int^{1}_{0} \Bigl(1-e^{-\frac{1}{2}y^2}\Bigr) \frac{1}{y} \, dy\\
 &\lesssim \Bigl(\log\sqrt{2n+l}\Bigr)^{\frac{2}{s} - 1},
\end{align*}
where we use the fact that
$$\int^{1}_{0}\Bigl(1-e^{-\frac{1}{2}y^2}\Bigr)\frac{1}{y} \, dy<+\infty.$$
We estimate the second term $A_2$. 
From the classical inequality  $|P_l|\leq1$ on $[-1,1]$,
\begin{align*}
|A_{2}|& \leq
\int^{\frac{\sqrt{2}}{2}}_{0}
(\log x^{-1})^{\frac{2}{s}-1} \,
|P_{l}(x)| \, x^{2n+l} \frac{dx}{x}
\\
&\leq
\left(\frac{\sqrt{2}}{2}\right)^{2n+l-1}
\int^{\frac{\sqrt{2}}{2}}_{0}
(\log x^{-1})^{\frac{2}{s}-1} \, dx\\
&\leq\left(\frac{\sqrt{2}}{2}\right)^{2n+l-1}
\int^{+\infty}_{\sqrt{2}}
x^{\frac{2}{s}-1}e^{-x} \, dx\leq \Gamma(\frac{2}{s}).
\end{align*}
We estimate the third term $A_{3}$.
We divide $A_3$ into two parts for $l\geq 2$
\begin{align*}
A_3 &=
\int^{\frac{\sqrt{2}}{2}}_{0}
(\log x^{-1})^{\frac{2}{s}-1} \,
\Bigl(1-P_{l}\bigl(\sqrt{1-x^2}\bigr)\Bigr) \, (1-x^2)^{\frac{2n+l}{2}}
\frac{dx}{x}
\\
&= \int^{\frac{\sqrt{2}}{2}}_{\frac{1}{l}}
+\int^{\frac{1}{l}}_0
=A_{31}+A_{32}.
\end{align*}
For the first part $A_{31}$, since $|P_l|\leq1$ on $[-1,1]$, 
we can estimate as follows
\begin{align*}
0\leq A_{31}
&\leq
\int^{\frac{\sqrt{2}}{2}}_{\frac{1}{l}}
(\log x^{-1})^{\frac{2}{s}-1} \,
\frac{dx}{x}  \\
&\leq \frac{s}{2} \left(
(\log l)^{\frac{2}{s}}-(\log \sqrt{2})^{\frac{2}{s}} \right).
\end{align*}
For the second part $A_{32}$, setting $y= l \, x$, we get
\begin{align*}
0\leq A_{32}&\leq
\int^{1}_{0}
\Bigl(\log \frac{l}{y}\Bigr)^{\frac{2}{s}-1} \,
\Bigl(1 -
     P_{l}\Bigl(\sqrt{1-y^2/l^2}\Bigr)
\Bigr) \,
\frac{dy}{y}.
\end{align*}
By Lemma 2.3 in \cite{HAOLI} 
(for a proof, see lemma \ref{lem estim Pl} in the appendix),
we have
\[1-P_{l}\Bigl(\cos\frac{\theta}{l}\Bigr) = O( \theta^2) \]
uniformly for $l\geq 1$ and $\theta\in [0, \frac{\pi}{2}]$.
We then deduce that for $l\geq 2$
\begin{align*}
A_{32}
&\lesssim
\int^{1}_{0}
\Bigl(\log \frac{l}{y}\Bigr)^{\frac{2}{s}-1} \, y \, dy
\\
&\lesssim
(\log l)^{\frac{2}{s}-1}
\int^{1}_{0}
\Bigl(1+ \frac{1}{\log l}\log\frac{1}{y}\Bigr)^{\frac{2}{s}-1} \, y \, dy
\\
&\lesssim
(\log l)^{\frac{2}{s}-1}.
\end{align*}
It follows from the above estimate of $A_1$, $A_2$, $A_3$ and \eqref{dec}
$$(\log(2n+l+e))^{\frac{2}{s}}\lesssim\lambda_{n,l}=
A_1+A_{2}+A_{3}\lesssim (\log(2n+l+e))^{\frac{2}{s}}.$$
This concludes the proof of \eqref{c-0}.
\end{proof}

\section{Properties of some distribution spaces}\label{S3}

In order to give some more precise descriptions 
on the regularity of the li\-nea\-rized Boltzmann operator $\mathcal{L}$,
we introduce the following Sobolev-type spaces:
for any real numbers $\tau>0$ and $\nu>0$,
$$E^{\tau}_{\nu}(\mathbb{R}^3)=\left\{u\in C^\infty(\mathbb{R}^3),\,\,
\|e^{\tau \, \left(\log(e+\mathcal{H})\right)^{\frac{2}{\nu}}}u\|^2_{L^2}<+\infty\right\},$$
which is a Banach space with the norm 
$$\|u\|^2_{E^{\tau}_{\nu}(\mathbb{R}^3)}=
\|e^{\tau \, \left(\log(e+\mathcal{H})\right)^{\frac{2}{\nu}}}u\|^2_{L^2}.$$
In the case $\nu=2$, the space
$E^{\tau}_{2}(\mathbb{R}^3)$ is equivalent 
to the symmetric weighted Sobolev space $Q^{2\tau}(\mathbb{R}^3)$
with the norm (see (2.1) in \cite{GPR} or 25.3 of Ch. IV in \cite{Shubin})
\[\|u\|_{Q^{2\tau}(\mathbb{R}^3)} =
\Bigl\|\Bigl(e+\mathcal{H}\Bigr)^{\tau} \, u\Bigr\|_{L^2}.\]
\\
In the case $0<\nu<2$ and $\tau>0$, one can verify that 
$E^{\tau}_{\nu}(\mathbb{R}^3)$ is an intermediate space between 
the symmetric weighted Sobolev spaces and the Gelfand-Shilov spaces :
For all $\nu_1\geq\frac12$,
\[
S^{\nu_1}_{\nu_1}(\mathbb{R}^3) \subset
  E^{\tau}_{\nu}(\mathbb{R}^3)   \subset
  {\bigcap_{k\geq 0} Q^{k}(\mathbb{R}^3).}
\]
Moreover, we have the following property of this space:
There exists a constant $C=C_{\nu}>0$ such that 
(see Proposition \ref{sobolev-type})
\[
\forall k\geq 1,\,\,\forall \tau>0,\quad
\|u\|_{Q^{k}(\mathbb{R}^3)}
\leq
e^{C \, {\big(\frac{1}{\tau}\big)}^{\frac{\nu}{2-\nu}} \, k^{\frac{2}{2-\nu}}} \,
\|u\|_{E^{\tau}_{\nu}(\mathbb{R}^3)}.
\]
We denote the dual space of the Sobolev-type space $E^{\tau}_{\nu}(\mathbb{R}^3)$ by
$$\left(E^{\tau}_{\nu}(\mathbb{R}^3)\right)'=E^{-\tau}_{\nu}(\mathbb{R}^3).$$
Obviously, for $\nu>0$ fixed, when
$\tau_2>\tau_1$, one has
$E^{\tau_2}_{\nu}(\mathbb{R}^3)\subset E^{\tau_1}_{\nu}(\mathbb{R}^3);$
for $\tau$ fixed, when $\nu_2>\nu_1>0, $ one has
$E^{\tau}_{\nu_1}(\mathbb{R}^3)\subset E^{\tau}_{\nu_2}(\mathbb{R}^3).$
Then we have the following inclusions.
\begin{theorem}\label{relation}
Assume the cross section kernel $b$ of the linearized Boltzmann operator $\mathcal{L}$ defined in \eqref{b} with $s>0$.
For $\tau>0$, there exist constants $\tau_2>\tau_1>0$, such that
$$
E^{\tau\tau_2}_{s}(\mathbb{R}^3)\subset D_{\tau}(\mathcal{L})\subset E^{\tau\tau_1}_s(\mathbb{R}^3).
$$
In particular, when $0<s\leq2$, for $u\in D_1(\mathcal{L})$,
\begin{equation}\label{relation2}
\| u \|_{Q^{2\tau_1}} =
\|(e+\mathcal{H})^{\tau_1}u\|_{L^2}\leq \|e^{\tau_1 \, \left(\log(e+\mathcal{H})\right)^{\frac{2}{s}}}u\|_{L^2} <+\infty.
\end{equation}
\end{theorem}

\begin{remark}
By the conjugation property, we have
$$
E^{-\tau\tau_1}_{s}(\mathbb{R}^3)  \subset \left(D_{\tau}(\mathcal{L})\right)'
  \subset  E^{-\tau\tau_2}_s(\mathbb{R}^3).
$$
\end{remark}
Using the spectral decomposition,
we give some another expressions of the norm of
the spaces $D_{\tau}(\mathcal{L})$, $\left(D_{\tau}(\mathcal{L})\right)'$
and $E_{s}^{\tau}(\mathbb{R}^3)$.
\begin{proposition}\label{character}
Let us define for $n,l\in \mathbb{N}$
\begin{equation}\label{lambda tild}
\widetilde{\lambda}_{n,l} =
\left\{
\begin{aligned}
  &1 \quad n+l\leq 1,\\
  &\lambda_{n,l} \quad n+l \geq 2 .
\end{aligned}
\right.
\end{equation}
$(1)$   For $\tau>0$ and $u\in D_{\tau}(\mathcal{L})$, 
let $( u,\varphi_{n,l,m})$ be the inner product in $L^2(\mathbb{R}^3)$.
Therefore the sequence $\{( u,\varphi_{n,l,m})\}$ satisfies
\begin{equation*}
\|u\|^2_{D_{\tau}(\mathcal{L})}
= \sum^{+\infty}_{n=0}\sum^{+\infty}_{l=0}
\sum_{|m|\leq l}e^{\tau\widetilde{\lambda}_{n,l}}( u,\varphi_{n,l,m})^2<+\infty.
\end{equation*}
(2) For $\tau>0$ and $T\in \left(D_{\tau}(\mathcal{L})\right)'$, 
let $\langle T, \varphi_{n,l,m}\rangle$ be the inner product in sense of distribution.
Ana\-lo\-gously the sequence $\{\langle T, \varphi_{n,l,m}\rangle\}$ satisfies
\begin{equation*}
\|u\|^2_{(D_{\tau}(\mathcal{L}))'} =
\sum^{+\infty}_{n=0}\sum^{+\infty}_{l=0}\sum_{|m|\leq l}e^{-\tau\widetilde{\lambda}_{n,l}}
   \langle T, \varphi_{n,l,m}\rangle^2<+\infty.
\end{equation*}
$(3)$  For $\tau\in\mathbb{R}$,  $s>0$, let $u\in E_{s}^{\tau}(\mathbb{R}^3)$.
Therefore the sequence $\{\langle u,\varphi_{n,l,m}\rangle\}$ satisfies
\begin{equation*}
\|u\|^2_{E^{\tau}_{\nu}(\mathbb{R}^3)}=
 \sum^{+\infty}_{n=0}\sum^{+\infty}_{l=0}\sum_{|m|\leq l}e^{2\tau(\log(2n+l+\frac 32+e))^{\frac{2}{s}}}
 \langle u,\varphi_{n,l,m}\rangle^2<+\infty.
\end{equation*}
\end{proposition}
\begin{proof}
The proof of part $(1)$ is direct:
From the spectral decomposition \eqref{Lg=Sum}
\begin{align*}
\|u\|^2_{D_{\tau}(\mathcal{L})}
&=\sum^{+\infty}_{k=0}\tau^k(k!)^{-1}\|(\mathcal{L}+\mathbf{P})^\frac{k}{2}u\|^2_{L^2}
\\
&= \sum^{+\infty}_{k=0}\tau^k(k!)^{-1}
\sum^{+\infty}_{n=0}\sum^{+\infty}_{l=0}\sum_{|m|\leq l}
 \widetilde{\lambda}_{n,l}^k ( u,\varphi_{n,l,m})^2.
\end{align*}
Now we prove part $(2)$: 
From the above characterization of $D_{\tau}(\mathcal{L})$, we find that
$\varphi_{n,l,m}\in D_{\tau}(\mathcal{L})$.\,\,Then for any $T\in \big(D_{\tau}(\mathcal{L})\big)'$, $\langle T,\varphi_{n,l,m}\rangle$ is well-defined.\,\,
We cons\-truct the following smooth function
$$u_{N}(v)=\frac{1}{C_{N}}\sum_{\substack{2n+l\leq N\\n\geq0,l\geq0}}\sum_{|m|\leq l}e^{-\tau\widetilde{\lambda}_{n,l}}\langle T,\varphi_{n,l,m}\rangle\varphi_{n,l,m}(v),$$
where
$$C_{N}=\sqrt{\sum_{\substack{2n+l\leq N\\n\geq0,l\geq0}}\sum_{|m|\leq l}e^{-\tau\widetilde{\lambda}_{n,l}}\langle T,\varphi_{n,l,m}\rangle^2}.$$
Applying the result of $(1)$, we find that
$$\|u_{N}\|_{D_{\tau}(\mathcal{L})}=1.$$
From the definition of $\big(D_{\tau}(\mathcal{L})\big)'$, one can verify
$$\|T\|_{\big(D_{\tau}(\mathcal{L})\big)'}\geq\langle T,u_{N}\rangle=\sqrt{\sum_{\substack{2n+l\leq N\\n\geq0,l\geq0}}\sum_{|m|\leq l}e^{-\tau\widetilde{\lambda}_{n,l}}\langle T,\varphi_{n,l,m}\rangle^2}.$$
Passing $N$ to $+\infty$, we have
\begin{equation}\label{low}
\|T\|_{\big(D_{\tau}(\mathcal{L})\big)'}\geq\sqrt{\sum_{n\geq0}\sum_{l\geq0}\sum_{|m|\leq l}e^{-\tau\widetilde{\lambda}_{n,l}}\langle T,\varphi_{n,l,m}\rangle^2}.
\end{equation}
On the other hand, for any $u\in D_{\tau}(\mathcal{L})$ with $\|u\|_{D_{\tau}(\mathcal{L})}$=1, we define a series
$$u_N=\sum_{\substack{2n+l\leq N\\n\geq0,l\geq0}}\sum_{|m|\leq l}( u,\varphi_{n,l,m})\varphi_{n,l,m}.$$
Then $u_N\rightarrow u\in D_{\tau}(\mathcal{L})\,\,\text{as}\,N\rightarrow+\infty$ with $\|u_N\|_{D_{\tau}(\mathcal{L})}\leq1$.\,\,
Therefore,
\begin{align*}
|\langle T,u_N\rangle|&=|\sum_{\substack{2n+l\leq N\\n\geq0,l\geq0}}\sum_{|m|\leq l}( u,\varphi_{n,l,m})\langle T,\varphi_{n,l,m}\rangle|\\
&\leq\Big(\sum_{\substack{2n+l\leq N\\n\geq0,l\geq0}}\sum_{|m|\leq l}e^{-\tau\widetilde{\lambda}_{n,l}}\langle T,\varphi_{n,l,m}\rangle^2\Big)^\frac{1}{2}\|u_N\|_{D_{\tau}(\mathcal{L})}\\
&\leq\Big(\sum_{\substack{2n+l\leq N\\n\geq0,l\geq0}}\sum_{|m|\leq l}e^{-\tau\widetilde{\lambda}_{n,l}}\langle T,\varphi_{n,l,m}\rangle^2\Big)^\frac{1}{2}.
\end{align*}
By continuity,
$$|\langle T,u\rangle|=\lim_{N\rightarrow+\infty}|\langle T,u_N\rangle|$$
and we obtain
\begin{equation}\label{up}
\|T\|_{\big(D_{\tau}(\mathcal{L})\big)'}\leq\Big(\sum_{n\geq0}\sum_{l\geq0}\sum_{|m|\leq l}e^{-\tau\widetilde{\lambda}_{n,l}}\langle T,\varphi_{n,l,m}\rangle^2\Big)^{1/2}.
\end{equation}
Part $(2)$ follows from \eqref{low} and \eqref{up}.
\\
For the part $(3)$, note that $\left\{\varphi_{n,l,m}(v)\right\}$ constitutes an orthonormal basis of $L^2(\mathbb{R}^3)$ composed of eigenvectors of the harmonic oscillator $\mathcal{H}$, then
\begin{equation*}
e^{\tau \, \left(\log(e+\mathcal{H})\right)^{\frac{2}{s}}}\varphi_{n, l, m}=e^{\tau(\log(2n+l+\frac 32+e))^{\frac{2}{s}}}\, \varphi_{n, l, m}.
\end{equation*}
This ends the proof of Proposition \ref{character}.
\end{proof}
\begin{remark}
For $\tau>0$ and $u\in D^+_{\tau}(\mathcal{L})$, 
let $( u,\varphi_{n,l,m})$ be the inner product in $L^2(\mathbb{R}^3)$.
Therefore the sequence $\{( u,\varphi_{n,l,m})\}$ satisfies
\begin{equation*}
\|u\|^2_{D^+_{\tau}(\mathcal{L})}
= \sum^{+\infty}_{n=0}\sum^{+\infty}_{l=0}\sum_{|m|\leq l}\widetilde{\lambda}_{n,l}e^{\tau\widetilde{\lambda}_{n,l}}( u,\varphi_{n,l,m})^2<+\infty.
\end{equation*}
For $\tau>0$ and $T\in \left(D^+_{\tau}(\mathcal{L})\right)'$, let $\langle T, \varphi_{n,l,m}\rangle$ be the inner product in sense of distribution.
Ana\-lo\-gously the sequence $\{\langle T, \varphi_{n,l,m}\rangle\}$ satisfies
\begin{equation*}
\|u\|^2_{(D^+_{\tau}(\mathcal{L}))'} =
\sum^{+\infty}_{n=0}\sum^{+\infty}_{l=0}\sum_{|m|\leq l}\frac{1}{\widetilde{\lambda}_{n,l}}e^{-\tau\widetilde{\lambda}_{n,l}}
   \langle T, \varphi_{n,l,m}\rangle^2<+\infty.
\end{equation*}
\end{remark}

We are prepared to prove Theorem \ref{relation}.
\begin{proof}[Proof of Theorem \ref{relation}:]
Applying Proposition \ref{regular}, for $\tau>0$, there exist constant $\tau_2>\tau_1>0$,   such that
\begin{align*}
\sum^{+\infty}_{n=0}\sum^{+\infty}_{l=0}\sum_{|m|\leq l}
  e^{2\tau\tau_1\left(\log(2n+l+\frac{3}{2}+e)\right)^{\frac{2}{s}}}
  (u,  &  \varphi_{n,l,m})^2
\leq
\sum^{+\infty}_{n=0} \sum^{+\infty}_{l=0}\sum_{|m|\leq l}
  e^{\tau\widetilde{\lambda}_{n,l}}(u,\varphi_{n,l,m})^2\\
&\leq
\sum^{+\infty}_{n=0}\sum^{+\infty}_{l=0}\sum_{|m|\leq l}
  e^{2\tau\tau_2\left(\log(2n+l+\frac{3}{2}+e)\right)^{\frac{2}{s}}}
  (u,\varphi_{n,l,m})^2.
\end{align*}
From Proposition \ref{character}, we have
$$E^{\tau\tau_2}_s(\mathbb{R}^3)\subset D_{\tau}(\mathcal{L})\subset E^{\tau\tau_1}_s(\mathbb{R}^3).$$
In addition, when $0<s\leq2$, for $u\in D_1(\mathcal{L})$,
$$e^{2\tau_1\left(\log(2n+l+e)\right)^{\frac{2}{s}}}\geq e^{2\tau_1\left(\log(2n+l+\frac{3}{2}+e)\right)}=(2n+l+\frac{3}{2}+e)^{2\tau_1}.$$
Therefore
\begin{align*}
\|(e+\mathcal{H})^{\tau_1}u\|^2_{L^2}&=\sum^{+\infty}_{n=0}\sum^{+\infty}_{l=0}\sum_{|m|\leq l}(2n+l+\frac{3}{2}+e)^{2\tau_1}( u,\varphi_{n,l,m})^2\\
&\leq\sum^{+\infty}_{n=0}\sum^{+\infty}_{l=0}\sum_{|m|\leq l}e^{2\tau_1\left(\log(2n+l+\frac{3}{2}+e)\right)^{\frac{2}{s}}}( u,\varphi_{n,l,m})^2\\
&=\|u\|^2_{E^{\tau_1}_s}<+\infty.
\end{align*}
This ends the proof of Theorem \ref{relation}.
\end{proof}

\section{Proof of Theorems \ref{trick}-\ref{trick2}}\label{S4}

Now we are prepared to prove Theorem \ref{trick}.
\begin{proof}[\bf{Proof of Theorem \ref{trick}}]
We proceed to treat the proof by the following four steps.

{\bf Step 1.}\,\,Construction of an auxiliary function $g_{N}$~with initial datum
which approximates $g_{0}\in \left(D_{\tau}(\mathcal{L})\right)'.$

\indent For $g_{0}\in \left(D_{\tau}(\mathcal{L})\right)'$, 
$\widetilde{\lambda}_{n,l}$ defined in \eqref{lambda tild},
we obtain from Proposition \ref{character}
\begin{equation}\label{l2-expansion}
\sum^{+\infty}_{n=0}\sum^{+\infty}_{l=0}\sum_{|m|\leq l}
e^{-\tau\widetilde{\lambda}_{n,l}}
\langle g_0, \varphi_{n,l,m}\rangle^2
=\|g_0\|^2_{\left(D_{\tau}(\mathcal{L})\right)'}<+\infty.
\end{equation}
For all $n,l\in\mathbb{N},~m\in\{-l,\cdots,l\}$, 
we consider the Cauchy problem associated with the ODEs
\begin{equation*}
\left\{ \begin{aligned}
         &\partial_{t}a_{n,l,m}(t)+\lambda_{n,l}a_{n,l,m}(t)=0,\\
                  &a_{n,l,m}(0)=\langle g_0, \varphi_{n,l,m}\rangle.
                          \end{aligned} \right.
                          \end{equation*}
Direct calculation shows that $a_{n,l,m}(t)=e^{-\lambda_{n,l}t}\langle g_0, \varphi_{n,l,m}\rangle$.
Let us now fix some positive integer $N\geq3$ and define
the following function
$g_{N}:[0,+\infty[\times\mathbb{R}^{3}\rightarrow L^2(\mathbb{R}^3)$ by
\begin{align*}
g_{N}=\sum_{\substack{2n+l\leq N\\n\geq0,l\geq0}}\sum_{|m|\leq l}e^{-\lambda_{n,l}t}\langle g_0, \varphi_{n,l,m}\rangle \varphi_{n,l,m}.
\end{align*}
Then $g_N$ satisfies
\begin{equation}\label{eq-3}
\left\{ \begin{aligned}
&\partial_{t}g_N+\mathcal{L}g_N=0,\,\,\\
&g_N(0)=\sum_{\substack{2n+l\leq N\\n\geq0,l\geq0}}\langle g_0, \varphi_{n,l,m}\rangle  \varphi_{n,l,m}.
\end{aligned} \right.
\end{equation}

\smallskip
{\bf Step 2.} Existence of the solution to the Cauchy problem \eqref{eq-1}.\\
\indent It is obvious that,
$$\mathbf{P}g_N\equiv\mathbf{P}g_0.$$
For $N\in\mathbb{N}$ big enough and for any $P\in\mathbb{N}^+,$
\begin{align*}
\forall t>0,\,\,\|g_{N+P}-g_N\|^2_{\left(D_{\tau}(\mathcal{L})\right)'}
&=
\sum_{\substack{N+1\leq2n+l\leq N+P\\n\geq0,l\geq0}}
\sum_{|m|\leq\,l}e^{-2\lambda_{n,l}t}
  e^{-\lambda_{n,l}\tau}|\langle g_0, \varphi_{n,l,m}\rangle|^2\\
&\leq
\sum_{\substack{N+1\leq2n+l\leq N+P\\n\geq0,l\geq0}}
\sum_{|m|\leq\,l}
  e^{-\lambda_{n,l}\tau}|\langle g_0, \varphi_{n,l,m}\rangle|^2
\rightarrow0.
\end{align*}
Let us fix $T>0$, $N\geq3$ big enough.  
Using the estimate in Proposition \ref{regular}, we can check that
\begin{align*}
&\int^T_0
  \|g_{N+P}-g_N\|^2_{\left(D^+_{\tau}(\mathcal{L})\right)'}dt \\
=&
\sum_{\substack{N+1\leq2n+l\leq N+P\\n\geq0,l\geq0}}
\sum_{|m|\leq\,l}
\lambda_{n,l}^{-1}e^{-\lambda_{n,l}\tau}|\langle g_0, \varphi_{n,l,m}\rangle|^2\frac{1}{2\lambda_{n,l}}(1-e^{-2\lambda_{n,l}T})\\
\lesssim&
\sum_{\substack{N+1\leq2n+l\leq N+P\\n\geq0,l\geq0}}
\sum_{|m|\leq\,l}
  e^{-\lambda_{n,l}\tau}|\langle g_0, \varphi_{n,l,m}\rangle|^2\rightarrow0;
\end{align*}
\begin{align*}
&\int^T_0
  \|\partial_tg_{N+P}-\partial_tg_N\|^2_{\left(D^+_{\tau}(\mathcal{L})\right)'}dt\\
=&
\sum_{\substack{N+1\leq2n+l\leq N+P\\n\geq0,l\geq0}}
\sum_{|m|\leq\,l}
  e^{-\lambda_{n,l}\tau}|\langle g_0, \varphi_{n,l,m}\rangle|^2\frac{\lambda_{n,l}}{2\lambda_{n,l}}(1-e^{-2\lambda_{n,l}T})\\
\leq&
\frac{1}{2}
\sum_{\substack{N+1\leq2n+l\leq N+P\\n\geq0,l\geq0}}
\sum_{|m|\leq\,l}
  e^{-\lambda_{n,l}\tau}|\langle g_0, \varphi_{n,l,m}\rangle|^2\rightarrow0;
\end{align*}
\begin{align*}
&\int^T_0
\|\mathcal{L}^{\frac{1}{2}}g_{N+P}-\mathcal{L}^{\frac{1}{2}}g_N\|^2_{\left(D_{\tau}(\mathcal{L})\right)'}dt\\
\leq&\frac{1}{2}
\sum_{\substack{N+1\leq2n+l\leq N+P\\n\geq0,l\geq0}}
\sum_{|m|\leq\,l}
  e^{-\lambda_{n,l}\tau}|\langle g_0, \varphi_{n,l,m}\rangle|^2\rightarrow0.
\end{align*}
Therefore, for $T>0$ fixed,
\begin{align*}
&\{g_N(t)\}\,\, \text{is a Cauchy sequence in} \,\,\left(D_{\tau}(\mathcal{L})\right)',
\,\, \forall t \in [0,T];\\
&\{g_N\}\,\, \text{is a Cauchy sequence in} \,\,H^1([0,T], \left(D^+_{\tau}(\mathcal{L})\right)');\\
&\{\mathcal{L}^{\frac{1}{2}}g_N\} \,\, \text{is a Cauchy sequence in} \,\,L^{2}([0,T], \left(D_{\tau}(\mathcal{L})\right)').
\end{align*}
Then there exists a function $g\in L^{\infty}([0,T], \left(D_{\tau}(\mathcal{L})\right)')\bigcap H^1([0,T],\left(D^+_{\tau}(\mathcal{L})\right)')$, $\mathbf{P}g\equiv\mathbf{P}g_0$ and $\mathcal{L}^{\frac{1}{2}}g\in L^{2}([0,T], \left(D_{\tau}(\mathcal{L})\right)') $, such that
\begin{align*}
&\forall t>0,\,\,g_N(t)\rightarrow g(t)\,\,\,\text{in}\, \left(D_{\tau}(\mathcal{L})\right)',\\
&g_N\rightarrow g\,\,\,\text{in}\, H^1([0,T],\,\,\left(D^+_{\tau}(\mathcal{L})\right)'),\\
&\mathcal{L}^{\frac{1}{2}}g_N\rightarrow \mathcal{L}^{\frac{1}{2}}g\,\,\,\,\text{in}\,L^{2}([0,T], \left(D_{\tau}(\mathcal{L})\right)').
\end{align*}
By Sobolev embedding theorem
$$H^1([0,T],\,\,\left(D^+_{\tau}(\mathcal{L})\right)')\hookrightarrow\,C([0,T],\left(D^+_{\tau}(\mathcal{L})\right)'),$$
we have
$$g\in C([0,T],\left(D^+_{\tau}(\mathcal{L})\right)').$$
Now we prove $g$~is the desired weak solution of Cauchy problem \eqref{eq-1}.\,\,
For any test function $\phi\in C^1([0,T],C_{0}^{\infty}(\mathbb{R}^{3}))$ and $0<t<T$,~recalled from \eqref{eq-3},~we have
\begin{align*}
&\langle g_N(t),\phi(t)\rangle-\langle g_N(0),\phi(0)\rangle\\
&=\int^{t}_{0}\langle g_N(\tau),\partial_{\tau}\phi\rangle d\tau-\int^{t}_{0}\langle g_N(\tau),\mathcal{L}\phi(\tau)\rangle d\tau.\nonumber
\end{align*}
Passing to the limit as $N\rightarrow+\infty$,~we get
\begin{align}\label{equation}
&\langle g(t),\phi(t)\rangle-\langle g(0),\phi(0)\rangle\nonumber\\
&=\int^{t}_{0}\langle g,\partial_{\tau}\phi\rangle d\tau-\int^{t}_{0}\langle g,\mathcal{L}\phi\rangle d\tau.
\end{align}
Besides,~from \eqref{l2-expansion}, we see that
\begin{equation*}
g_N(0)\rightarrow g(0)\, \text{in}\,\left(D_{\tau}(\mathcal{L})\right)'.
\end{equation*}
Henceforth,~we obtain
$$\langle g(0),\phi(0)\rangle=\langle g_{0},\phi(0)\rangle.$$
Substituting the above result into \eqref{equation}, \eqref{def} follows.

\smallskip
{\bf Step 3.} Uniqueness of the solution to the Cauchy problem \eqref{eq-1}.\\
\indent Assume that $\tilde{g}$ is another solution satisfies \eqref{solution} and \eqref{def}.\,\,Denote
$$h(t)=g(t)-\tilde{g}(t).$$
For $T>0$, for any $\phi\in C^1([0,T],C_{0}^\infty(\mathbb{R}^{3}))$ and $0<t<T$, we have $h(0)=0$ and
\begin{equation}\label{unique}
\langle h(t),\phi(t)\rangle=\int^{t}_{0}\langle h,\partial_{\tau}\phi\rangle d\tau
-\int^{t}_{0}\langle h,\mathcal{L}\phi(\tau)\rangle d\tau.
\end{equation}
We define a smooth function
$$\phi(t)=\sum_{\substack{2n+l\leq N\\n\geq0,l\geq0}}e^{-2\tau\widetilde{\lambda}_{n,l}}\langle h(t), \varphi_{n,l,m}\rangle \varphi_{n,l,m}.$$
Substituted into \eqref{unique} as a test function, we have
$$
\sum_{\substack{2n+l\leq N\\n\geq0,l\geq0}}e^{-2\tau\widetilde{\lambda}_{n,l}}\langle h(t),\varphi_{n,l,m}\rangle^2\leq0.
$$
Passing $N\rightarrow+\infty$,~we have
$$\|h(t)\|_{\left(D_{\tau}(\mathcal{L})\right)'}^2\leq0.$$
Thus $h=0\,\,in\,\, L^2(\mathbb{R}^3)$.

\smallskip
{\bf Step 4.} Regularity of the solution with the initial data $g_0\in E^{-t_0}_s(\mathbb{R}^3)$.\\
\indent
Using Theorem \ref{relation}, we have
$$D_{t_0/\tau_1}(\mathcal{L})\subset E^{t_0}_s(\mathbb{R}^3)
\quad\text{and}\quad
E^{-t_0}_s(\mathbb{R}^3)\subset \left(D_{t_0/\tau_1}(\mathcal{L})\right)$$
where $\tau_1$ is given by Theorem \ref{relation}.
Namely,
$$\|e^{- t_0 \, \left(\log(e+\mathcal{H})\right)^{\frac{2}{s}}}g_0 \|_{L^2} < +\infty.$$
Recalling that $g_0=\mathbf{P}g_0+(\mathbf{I}-\mathbf{P})g_0$, we have
\begin{align}
(\mathbf{I}-\mathbf{P})g_N(t)=\sum_{\substack{2n+l\leq N\\n\geq0,l\geq0}}e^{-\lambda_{n,l}t}\langle (\mathbf{I}-\mathbf{P})g_0,\varphi_{n,l,m}\rangle\varphi_{n,l,m}.\nonumber
\end{align}
Then for any $T>0$,
$$(\mathbf{I}-\mathbf{P})g_N(t)\rightarrow g\,\,\,\text{in}\, L^{\infty}([0,T], \left(D_{t_0/\tau_1}(\mathcal{L})\right)').$$
Moreover, from Proposition \ref{regular},
there exists a constant $c>0$ such that for any $t>t_0/c$
\begin{align*}
\|e^{ct\left(\log(\mathcal{H}+e)\right)^{\frac{2}{s}}}(\mathbf{I}-\mathbf{P})g_N\|_{L^2(\mathbb{R}^3)}
\leq e^{-\frac{1}{4}\lambda_{2,0}t}\|(\mathbf{I}-\mathbf{P})g_{0}\|_{E^{-t_0}_s(\mathbb{R}^3)},
\end{align*}
where we have used the following estimate of the eigenvalue
(see Part 4.3 of \cite{GLX_2015})
$$0<\lambda_{2,0}\leq \lambda_{n,l}.$$
By the lower continuity, we have
\begin{align}\label{final}
&\|e^{ct\left(\log(\mathcal{H}+e)\right)^{\frac{2}{s}}}(\mathbf{I}-\mathbf{P})g\|_{L^2(\mathbb{R}^3)}\nonumber\\
&\leq\liminf_{N\rightarrow+\infty}\|e^{ct\left(\log(\mathcal{H}+e)\right)^{\frac{2}{s}}}(\mathbf{I}-\mathbf{P})g_N\|_{L^2(\mathbb{R}^3)}\nonumber\\
&\leq  e^{-\frac{1}{4}\lambda_{2,0}t}\|(\mathbf{I}-\mathbf{P})g_{0}\|_{E^{-t_0}_s(\mathbb{R}^3)}.
\end{align}
This concludes the proof of Theorem \ref{trick}.\\
\end{proof}

We now give the proof of Theorem \ref{trick2}

\begin{proof}
Consider that $g_0\in L^2(\mathbb{R}^3)$. 
For $0<s\leq2$, by using the formula \eqref{relation2} in Theorem \ref{relation} 
and the formula \eqref{final}, we obtain,
\begin{align*}
\|\left(\mathcal{H}+e\right)^{ct}(\mathbf{I}-\mathbf{P})g\|_{L^2(\mathbb{R}^3)}\leq  e^{-\frac{1}{4}\lambda_{2,0}t}\|(\mathbf{I}-\mathbf{P})g_{0}\|_{L^2(\mathbb{R}^3)}.
\end{align*}
This is the formula \eqref{rate1}.

For the part $2)$ of Theorem \ref{trick2}, i.e, in the case $0<s<2$,
the formula \eqref{rate2} follows from Proposition \ref{sobolev-type} and formula \eqref{final}.
The proof of Theorem \ref{trick2} is completed.
\end{proof}

Now we study the regularizing effect of the Boltzmann equation 
with the intial data in Sobolev space, 
which provides a detailed exposition of Remark \ref{remark}.

Recalled that  $\left\{\varphi_{n,l,m}(v)\right\}_{n,l\in \mathbb{N},|m|\leq l}$ constitutes an orthonormal basis of $L^2(\mathbb{R}^3)$ composed of eigenvectors of the harmonic oscillator
(see\cite{Boby}, \cite{NYKC2})
\begin{equation*}
\mathcal{H}(\varphi_{n, l, m})=(2n+l+\frac 32)\, \varphi_{n, l, m}.
\end{equation*}
For $\tau>0$ (see \eqref{Q2tau} in the Appendix),
$$\|u\|^2_{Q^{\tau}(\mathbb{R}^3)}=\|(\mathcal{H}+e)^{\frac{\tau}{2}}u\|^2_{L^2}=\sum^{+\infty}_{n=0}\sum^{+\infty}_{l=0}\sum_{|m|\leq l}(2n+l+\frac{3}{2}+e)^{\tau}(u,\varphi_{n, l, m})^2_{L^2}.$$
We consider the example in Remark \ref{remark}.

In the case $s=2$, the regularizing effect in $Q^{k}(\mathbb{R}^3)$ has usually a positive time delay.
\begin{Example}
Consider the initial data
$g_0=\sum_{n\geq 2} \frac{1}{n^{\frac12} \log n}   \phi_{n,0,0}$, then
\begin{align*}
\|g_0\|_{L^2(\mathbb{R}^3)}^2=\sum^{+\infty}_{n=2}\frac{1}{n(\log n)^2}<+\infty.
\end{align*}
This shows that $g_0\in L^2(\mathbb{R}^3)$.
The solution of the Cauchy problem \eqref{eq-1} can be written as
\begin{align*}
g(t)=\sum^{+\infty}_{n=2}e^{-t\lambda_{n,0}}\frac{1}{n^{\frac12} \log n}   \phi_{n,0,0}.
\end{align*}
By using Proposition \ref{regular}, for $s=2$, we have
$$\lambda_{n,0}\approx \log(2n+e).$$
Then
\begin{align*}
\|g(t)\|^2_{Q^k(\mathbb{R}^3)}&=\sum^{+\infty}_{n=2}(2n+\frac{3}{2}+e)^{k}e^{-2t\lambda_{n,0}}\frac{1}{n(\log n)^2}\\
&\approx\sum^{+\infty}_{n=2}(2n+\frac{3}{2}+e)^{k}\frac{1}{n^{1+2t}(\log n)^2},
\end{align*}
which is convergent when $2t\geq k$, i.e. $t\geq\frac{k}{2}$.

We can check that there exists $t_k=\frac{k}{2}$ such that $g(t)\not\in Q^{k}(\mathbb{R}^3)$ for $0\leq t<t_k$.
\end{Example}
In the case $s>2$. We prove that there is no regularizing effect in the Sobolev space.
\begin{Example}
Consider any real numbers $0<\tau<\tau'$ and
$g_0=\sum_{n\geq 2} \frac{1}{n^{\frac{\tau+1}2} \log n} \phi_{n,0,0}$.
Therefore we have
\begin{align*}
\|g_0\|^2_{Q^{\tau}(\mathbb{R}^3)}&=\sum^{+\infty}_{n=2}(2n+\frac{3}{2}+e)^{\tau}\frac{1}{n^{\tau+1} (\log n)^2}\\
&\approx\sum^{+\infty}_{n=2}\frac{1}{n (\log n)^2}<+\infty.
\end{align*}
This means that $g_0\in Q^{\tau}(\mathbb{R}^3).$

However, the solution of the Cauchy problem \eqref{eq-1} can be written as
\begin{align*}
g(t)=\sum^{+\infty}_{n=2}e^{-t\lambda_{n,0}}\frac{1}{n^{\frac{\tau+1}2} \log n} \phi_{n,0,0}.
\end{align*}
It is easy to prove that $g(t)\in Q^{\tau}(\mathbb{R}^3).$  Now we prove $g(t)
\notin Q^{\tau'}(\mathbb{R}^3).$

In the case $s>2$, by using the Proposition \ref{regular}, we obtain
$$\lambda_{n,0}\approx (\log(2n+e))^{\frac{2}{s}}.$$
Therefore, there exists a constant $c_0>0$ such that
\begin{align*}
\|g(t)\|^2_{Q^{\tau'}(\mathbb{R}^3)}&=\sum^{+\infty}_{n=2}(2n+\frac{3}{2}+e)^{\tau'}e^{-2t\lambda_{n,0}}\frac{1}{n^{\tau+1} (\log n)^2}\\
&\leq\sum^{+\infty}_{n=2}\frac{n^{\tau'-\tau}}{e^{2tc_0(\log(2n+e))^{\frac{2}{s}}}n (\log n)^2}\\
&=\sum^{+\infty}_{n=2}e^{{\frac{\tau'-\tau}{2}\log (n)-2tc_0(\log(2n+e))^{\frac{2}{s}}}}\frac{1}{n^{1-\frac{\tau'-\tau}{2}} (\log n)^2}.
\end{align*}
Considering the condition that $\tau'>\tau$ and $s>2$, one can verify that
$${\frac{\tau'-\tau}{2}\log (n)-2tc_0(\log(2n+e))^{\frac{2}{s}}}\rightarrow+\infty,\,\,\text{as}\,\,n\rightarrow+\infty.$$
Since
$$\sum^{+\infty}_{n=2}\frac{1}{n^{1-\frac{\tau'-\tau}{2}} (\log n)^2}=+\infty,$$
we obtain
$$
\|g(t)\|^2_{Q^{\tau'}(\mathbb{R}^3)}=+\infty.$$
We conclude 
$$g(t)\notin Q^{\tau'}(\mathbb{R}^3).$$
We can check that for $t\geq0$ the solution $g(t)$ stays in the space $\in Q^{\tau}(\mathbb{R}^3)$,
but never belongs to $Q^{\tau'}(\mathbb{R}^3)$.
\end{Example}

\section{Appendix}\label{Appendix}

We present in this section some spectral properties
of the functional spaces used in this paper and the proof of lemma \ref{lem estim Pl}.\\

The symmetric Gelfand-Shilov space $S^{\nu}_{\nu}(\mathbb{R}^3)$ can be characterized through the decomposition
into the Hermite basis $\{H_{\alpha}\}_{\alpha\in\mathbb{N}^3}$ and the harmonic oscillator $\mathcal{H}=-\triangle +\frac{|v|^2}{4}$.
For more details, see Theorem 2.1 in the book \cite{GPR}.
\begin{align*}
f\in S^{\nu}_{\nu}(\mathbb{R}^3)
&\Leftrightarrow\,f\in C^\infty (\mathbb{R}^3),\,\,\exists\, \tau>0,
\bigl\|\,e^{\tau\mathcal{H}^{\frac{1}{2\nu}}}f\,\bigr\|_{L^2}<+\infty;
\\
&\Leftrightarrow\, f\in\,L^2(\mathbb{R}^3),\,\,\exists\, \epsilon_0>0,\,\,
\Bigl\|\,\Big(le^{\epsilon_0|\alpha|^{\frac{1}{2\nu}}}(f\,H_{\alpha})_{L^2}\Bigr)_{\alpha\in\mathbb{N}^3}\Bigr\|_{l^2}<+\infty;
\\
&\Leftrightarrow\,\exists\,C>0,\,A>0,\,\,
\Bigl\|\,\Bigl(-\triangle +\frac{|v|^2}{4}\Bigr)^{\frac{k}{2}}f\,\Bigr\|_{L^2(\mathbb{R}^3)}
\leq A\,C^k\,(k!)^{\nu},\,\,\,k\in\mathbb{N}
\end{align*}
where
$$H_{\alpha}(v)=H_{\alpha_1}(v_1)H_{\alpha_2}(v_2)H_{\alpha_3}(v_3),\,\,\alpha\in\mathbb{N}^3,$$
and for $x\in\mathbb{R}$,
$$H_{n}(x)=\frac{(-1)^n}{\sqrt{2^nn!\pi}}e^{\frac{x^2}{2}}\frac{d^n}{dx^n}(e^{-x^2})
=\frac{1}{\sqrt{2^nn!\pi}}\Big(x-\frac{d}{dx}\Big)^n(e^{-\frac{x^2}{2}}).$$
For the harmonic oscillator
$\mathcal{H}=-\triangle +\frac{|v|^2}{4}$ of 3-dimension and $s>0$,
we have
$$
\mathcal{H}^{\frac{k}{2}} H_{\alpha} = (\lambda_{\alpha})^{\frac{k}{2}}H_{\alpha},\,\, \lambda_{\alpha}=\sum^3_{j=1}(\alpha_j+\frac{1}{2}),\,\,k\in\mathbb{N},\,\alpha\in\mathbb{N}^3.
$$
The symmetric weighted Sobolev space $Q^{2\tau}(\mathbb{R}^3)$
can be also characterized through the decomposition
into the Hermite basis :
\begin{align}\label{Q2tau}
\begin{split}
f\in Q^{2\tau}(\mathbb{R}^3)
&\Leftrightarrow\,f\in\,L^2(\mathbb{R}^3),\,\,
\Bigl\| \, \Bigl(e+\mathcal{H}\Bigr)^{\tau} \, f \, \Bigr\|_{L^2}<+\infty;
\\
&\Leftrightarrow\, f\in\,L^2(\mathbb{R}^3),\,\,
\Big\|\,\Big(
|\alpha|^{\tau} (f,\,H_{\alpha})_{L^2}
\Big)_{\alpha\in\mathbb{N}^3}\Big\|_{l^2}<+\infty.
\end{split}
\end{align}
Concerning the Sobolev-type space
$E^{\tau}_{\nu}(\mathbb{R}^3)$ for $\nu>0$ introduced in part 1,
we have the following property :
\begin{proposition}\label{sobolev-type}
Let $0<\nu<2$ and $\tau>0$.
There exists a constant $C=C_{\nu}$ such that, 
for ant $f\in E^{\tau}_{\nu}(\mathbb{R}^3)$,
\[
\forall k\geq 1,\quad
\Bigl\|\,\Bigl(-\Delta+{\textstyle \frac{|v|^2}{4}}\Bigr)^{\frac{k}{2}} \, f\,\Bigr\|_{L^2}\
\leq
  \, e^{C \, \big(\frac{1}{\tau}\big)^{\frac{\nu}{2-\nu}}  \, k^{\frac{2}{2-\nu}}}
  \, \|\,f\,\|_{E^{\tau}_{\nu}}.
\]
\end{proposition}

\begin{proof}
We expand $f$ in the Hermite basis : noting $f_\alpha = (f,H_\alpha)_{L^2}$, 
we get
\begin{align*}
\sum_{\alpha}
e^{2\tau \, \left(\log(e+\lambda_\alpha)\right)^{\frac{2}{\nu}}} |f_\alpha|^2
= \|\,f\,\|_{E^{\tau}_{\nu}}^2.
\end{align*}
We rephrase the previous identity as follows
\begin{align*}
\sum_{\alpha\in\mathbb{N}^3}
h_{\tau,k}(e+\lambda_\alpha) \, (e+\lambda_\alpha)^k |f_\alpha|^2
= \|\,f\,\|_{E^{\tau}_{\nu}}^2
\end{align*}
where
$$h_{\tau,k}(x) = \frac{e^{2\tau\,\left(\log x\right)^{\frac{2}{\nu}}}}{x^k}.$$
It is easy to check that
\begin{equation}\label{young}
\forall x\geq1, \quad h_{\tau,k}(x) \geq
e^{-\frac{2-\nu}{2}\big(\frac{\nu}{4\tau}\big)^{\frac{\nu}{2-\nu}}k^{\frac{2}{2-\nu}}}.
\end{equation}
Indeed, using Young's inequality
$$xy\leq \frac{1}{p}x^p+\frac{1}{q}y^q,
\,\,\text{where}\,\,\frac{1}{p}+\frac{1}{q}=1,$$
with $p=\frac{2}{2-\nu}$, $q=\frac{2}{\nu}$,
we obtain
$$k \, \log x
\leq
\frac{2-\nu}{2} \, 
\Big[\Big(\frac{\nu}{4\tau}\Big)^{\frac{\nu}{2}} k\Big]^{\frac{2}{2-\nu}} \, 
+ \, 2\tau \, \bigl[\log x\bigr]^{\frac{2}{\nu}},$$
and \eqref{young} follows immediately.
Therefore we deduce that
\begin{align*}
\|\,f\,\|_{E^{\tau}_{\nu}}^2
&=
\sum_{\alpha\in\mathbb{N}^3}
h_{\tau,k} (e+\lambda_{\alpha})\, (e+\lambda_{\alpha})^k \, |f_{\alpha}|^2 \\
&\geq
e^{-\frac{2-\nu}{2}\big(\frac{\nu}{4\tau}\big)^{\frac{\nu}{2-\nu}}k^{\frac{2}{2-\nu}}} \,
\sum_{\alpha\in\mathbb{N}^3}
(e+\lambda_\alpha)^k \, |f_\alpha|^2
\end{align*}
and using \eqref{Q2tau} we conclude the proof.
\end{proof}

We now give a proof of Lemma 2.3 in \cite{HAOLI}.
\begin{lemma} \label{lem estim Pl}
 (\cite{HAOLI})
We have
\begin{equation*}
1 - P_{l}\Bigl(\cos\frac{\theta}{l}\Bigr) = O(\theta^2)
\end{equation*}
uniformly for $l\geq 1$ and $\theta\in [0, \frac{\pi}{2}]$.
\end{lemma}
\begin{proof}
We recall that the Legendre polynomial $w(t)=P_l(\cos t)$ satisfies
\[
\frac{1}{\sin t} \, \frac{d}{dt}\left(
\sin t \, \frac{d w}{dt}
\right) + l\,(l+1)\, w = 0,
\quad w(0)=1, \quad |w|\leq 1.
\]
Let us define the new function for $\theta\in [0, \frac{\pi}{2}]$
\[u(\theta) = 1 - P_{l}\Bigl(\cos\frac{\theta}{l}\Bigr). \]
The function $u$ is solution of the differential equation
\[
\frac{d}{d\theta}\left( \sin\left(\frac{\theta}{l}\right) \, \frac{d u}{d\theta} \right)
= \frac{l+1}{l} \, \sin\left(\frac{\theta}{l}\right) \, w
= O\left( \frac{\theta}{l} \right).
\]
We integrate on the interval $[0,\theta]$ and we get
\[
\sin\left(\frac{\theta}{l}\right) \, \frac{d u}{d\theta}
= O\left( \frac{\theta^2}{l} \right).
\]
Therefore
\[ \frac{d u}{d\theta} = O( \theta ). \]
Since $u(0)=0$, another integration finishes the proof of the estimate.
\end{proof}

\section*{\bf Acknowledgments}
The first author would like to express his sincere thanks to School of mathema\-tics and statistics of Wuhan University for the invitation.
The second author deeply {\color{black}thanks} Professor Chao-Jiang Xu for his valuable discussion all along this work.

\end{document}